\documentclass{amsart}%
\usepackage{color}
\usepackage{amsmath}
\usepackage{amsfonts}
\usepackage{amssymb}
\usepackage{graphicx}%
\providecommand{\U}[1]{\protect\rule{.1in}{.1in}}

\newtheorem{theorem}{Theorem}[section]

\newtheorem{lemma}[theorem]{Lemma}

\begin{document}
\title[Uniform estimates for the Schr\"odinger flow]{Uniform estimates for the solutions of the Schr\"odinger equation on the torus
and regularity of semiclassical measures}
\author[T. A\"issiou]{Tayeb A\"issiou}
\address{Department of Mathematics and Statistics, Concordia University, 1455 de
Maisonneuve Blvd. West, Montreal, Quebec, H3G 1M8, Ca\-na\-da.}
\email{aissiou@math.mcgill.ca}
\author[D. Jakobson]{Dmitry Jakobson}
\address{Department of Mathematics and Statistics, McGill University, 805 Sherbrooke
Str. West, Montr\'eal QC H3A 2K6, Ca\-na\-da.}
\email{jakobson@math.mcgill.ca}
\author[F. Maci\`a]{Fabricio Maci\`a}
\address{Universidad Polit\'ecnica de Madrid, DCAIN, ETSI Navales, Avda. Arco de la
Victoria s/n. 28040 Madrid, Spain}
\email{fabricio.macia@upm.es}
\thanks{T.A. was supported by FQRNT. D.J. was supported by NSERC, FQRNT and Dawson
fellowship. F.M. was supported by grant MTM2010-16467 (MEC), and wishes to
acknowledge the support of ICMAT through its visiting faculty program.}
\date{\today}

\begin{abstract}
We establish uniform bounds for the solutions $e^{it\Delta}u$ of the
Schr\"{o}\-din\-ger equation on arithmetic flat tori, generalising earlier
results by J. Bourgain.
We also study the regularity properties of weak-$*$ limits of sequences of
densities of the form $|e^{it\Delta}u_{n}|^{2}$ corresponding to highly
oscillating sequences of initial data $(u_{n})$. We obtain improved regularity
properties of those limits using previous results by N. Anantharaman and F.
Maci\`a on the structure of semiclassical measures for solutions to the
Schr\"{o}dinger equation on the torus.

\end{abstract}
\maketitle

\section{Introduction}

Consider the Schr\"{o}dinger flow $e^{it\Delta}$ on the arithmetic flat
$d$-dimensional torus $\mathbb{T}^{d}:=\mathbb{R}^{d}/2\pi\mathbb{Z}^{d}$.
That is, given any $u\in L^{2}\left(  \mathbb{T}^{d}\right)  $, the function
$\psi(x,t):=e^{it\Delta}u(x)$ is the solution to:%
\begin{equation}
i\partial_{t}\psi(x,t)+\Delta\psi(x,t)=0,\quad\psi(\cdot,0)=u. \label{def:psi}%
\end{equation}
In this article we are be interested in those regularity properties of $\psi$
that can be expressed through the quantity:%
\[
\left\vert \psi(x,t)\right\vert ^{2},
\]
which \emph{a priori}, is a continuous function of $t\in\mathbb{T}$ (as the
Schr\"{o}dinger flow is periodic) taking values in the set of positive
elements in $L^{1}\left(  \mathbb{T}^{d}\right)  $. By conservation of energy,
the total mass of $\left\vert \psi\left(  \cdot,t\right)  \right\vert ^{2}$
equals $\left\Vert u\right\Vert _{L^{2}\left(  \mathbb{T}^{d}\right)  }^{2}$
for every $t\in\mathbb{T}$.

Given $u\in L^{2}\left(  \mathbb{T}^{d}\right)  $, denote by $b_{u}\left(
l,s\right)  $, with $\left(  l,s\right)  \in\mathbb{Z}^{d}\times\mathbb{Z}$,
the Fourier coefficients of $\left\vert \psi\right\vert ^{2}$:
\[
\left\vert e^{it\Delta}u\left(  x\right)  \right\vert ^{2}=\sum_{\left(
l,s\right)  \in\mathbb{Z}^{d}\times\mathbb{Z}}b_{u}\left(  l,s\right)
e^{i\left(  l,s\right)  \cdot\left(  x,t\right)  };
\]
note that $b_{u}$ is a quadratric quantity in $u$. We shall prove the following.

\begin{theorem}
\label{t:ld}There exists a constant $C_{d}>0$ such that:%
\begin{equation}
\left\Vert b_{u}\right\Vert _{l^{d+1}\left(  \mathbb{Z}^{d+1}\right)  }\leq
C_{d}\left\Vert u\right\Vert _{L^{2}\left(  \mathbb{T}^{d}\right)  }^{2},
\label{e:genZyg}%
\end{equation}
for every $u\in L^{2}\left(  \mathbb{T}^{d}\right)  $.
\end{theorem}

When $d=1$, Theorem \ref{t:ld} has been proved by Bourgain
\cite{BourgainStrichartz93}, using the methods developed by Cooke
\cite{CookeTorus2d} and Zygmund \cite{Zygmund74}. Note that in that case,
estimate (\ref{e:genZyg}) is equivalent to:%
\begin{equation}
\left\Vert e^{it\partial_{x}^{2}}u\right\Vert _{L^{4}\left(  \mathbb{T}%
_{x}\times\mathbb{T}_{t}\right)  }\leq\left(  C_{1}\right)  ^{1/2}\left\Vert
u\right\Vert _{L^{2}\left(  \mathbb{T}\right)  }. \label{e:Stric}%
\end{equation}
When $d\geq2$, Bourgain has proved in \cite{BourgainStrichartz93, Bourgain11}
frequency-dependent generalizations of estimate (\ref{e:Stric}) that are
optimal in many cases. When the initial datum $u(x)$ is an eigenfunction of
$\Delta$:
\begin{equation}
\Delta u+\lambda u=0,\quad\left\Vert u\right\Vert _{L^{2}\left(
\mathbb{T}^{d}\right)  }=1; \label{e:eigenf}%
\end{equation}
then $e^{it\Delta}u=e^{-it\lambda}u$, which implies $|e^{it\Delta}%
u|^{2}=|u|^{2}$. For that class of initial data, the Fourier coefficients
$b_{u}\left(  l\right)  $ only depend on $l\in\mathbb{Z}^{d}$ and the exponent
in estimate (\ref{e:genZyg}) can be improved to:%
\begin{equation}
\left\Vert b_{u}\right\Vert _{l^{d}\left(  \mathbb{Z}^{d}\right)  }\leq K_{d},
\label{e:reigen}%
\end{equation}
where the constant $K_{d}$ is independent of the eigenvalue $\lambda$. This
result was established in \cite{CookeTorus2d, Zygmund74} for $d=2$ and in
\cite{JakobsonTori97, JakNadToth01, Aissiou} for $d\geq3$.

Theorem \ref{t:ld} shows that as a function of space and time, $e^{it\Delta}u$
is more regular than being merely a function in $C\left(  \mathbb{T}%
;L^{2}\left(  \mathbb{T}^{d}\right)  \right)  $, but this regularity is only
expressed in terms of the summability properties of the Fourier coefficients
of its modulus square. The proof of Theorem \ref{t:ld} relies on ideas in
\cite{JakobsonTori97}; it is presented in Section \ref{s:est}, along with a
straightforward generalisation to the density-matrix Schr\"{o}dinger equation
(Theorem \ref{t:lddm}). \bigskip

Next, we state our second result. Suppose $\left(  u_{n}\right)  $ is a
bounded sequence in $L^{2}\left(  \mathbb{T}^{d}\right)  $. Then the sequence%
\[
\left(  \left\vert e^{it\Delta}u_{n}\right\vert ^{2}\right)  \subset C\left(
\mathbb{T};L^{1}\left(  \mathbb{T}^{d}\right)  \right)  ,
\]
is uniformly bounded. It is always possible to extract a subsequence $\left(
u_{n^{\prime}}\right)  $ such that that the $\left\vert e^{it\Delta
}u_{n^{\prime}}\right\vert ^{2}$ converge to some positive measure $\nu\in
L^{\infty}\left(  \mathbb{T};\mathcal{M}_{+}\left(  \mathbb{T}^{d}\right)
\right)  $ in the weak-$\ast$ topology,\footnote{We have denoted by
$\mathcal{M}_{+}\left(  \mathbb{T}^{d}\right)  $ the set of positive Radon
measures on $\mathbb{T}^{d}$.} \emph{i.e.}%
\begin{equation}
\lim_{n^{\prime}\rightarrow\infty}\int_{\mathbb{T}\times\mathbb{T}^{d}}%
\phi\left(  t\right)  a\left(  x\right)  \left\vert e^{it\Delta}u_{n^{\prime}%
}\left(  x\right)  \right\vert ^{2}dtdx=\int_{\mathbb{T}}\int_{\mathbb{T}^{d}%
}\phi\left(  t\right)  a\left(  x\right)  \nu\left(  dx,t\right)  dt,
\label{e:wc}%
\end{equation}
for every $\phi\in L^{1}\left(  \mathbb{T}\right)  $ and every $a\in C\left(
\mathbb{T}^{d}\right)  $. In what follows, we shall say that a measure $\nu$
is a \emph{weak-}$\ast$ \emph{accumulation point }of the sequence of densities
$\left(  \left\vert e^{it\Delta}u_{n}\right\vert ^{2}\right)  $ if
(\ref{e:wc}) holds for some subsequence.

Our next result deals with the regularity properties of those measures $\nu$
that arise as weak-$\ast$ accumulation points of sequences of densities of
orbits of the Schr\"{o}dinger flow. A result by Bourgain \cite{BourgainQL97}
asserts that in fact all weak-$\ast$ accumulation points $\nu$ are absolutely
continuous with respect to the Lebesgue measure, \emph{i.e.} one always has
$\nu\in L^{\infty}\left(  \mathbb{T};L^{1}\left(  \mathbb{T}^{d}\right)
\right)  $. This result was further generalised by Anantharaman and Maci\`{a}
\cite{AnantharamanMacia} to the case of the Laplacian plus a potential;
moreover, these authors deduced a propagation law for $\nu\left(
\cdot,t\right)  $ and clarified the dependence of $\nu$ on the sequence of
initial data $\left(  u_{n}\right)  $ (these results generalise those in
\cite{MaciaTorus}).

In addition, because of estimate (\ref{e:genZyg}) one easily sees that such a
$\nu$ also satisfies $\left\Vert \widehat{\nu}\right\Vert _{l^{d+1}\left(
\mathbb{Z}^{d+1}\right)  }<\infty$, where $\widehat{\nu}\left(  l,s\right)  $,
$\left(  l,s\right)  \in\mathbb{Z}^{d}\times\mathbb{Z}$, stand for the Fourier
coefficients of $\nu$. These two facts express that the regularity properties
of $\nu$ are as good as those of the densities $\left\vert e^{it\Delta}%
u_{n}\right\vert ^{2}$. And this cannot be improved at this level of
generality, since if a sequence $\left(  u_{n}\right)  $ converges strongly to
some $u$ in $L^{2}\left(  \mathbb{T}^{d}\right)  $ then $\nu\left(
\cdot,t\right)  =\left\vert e^{it\Delta}u\right\vert ^{2}$.

However, the regularity of the weak-$\ast$ accumulation points can be improved
if we restrict ourselves to sequences $\left(  u_{n}\right)  $ that are highly
oscillating. This is for instance the case when $\left(  u_{n}\right)  $
consists of eigenfunctions of the Laplacian (\ref{e:eigenf}) corresponding to
eigenvalues $\lambda=\lambda_{n}\rightarrow\infty$ as $n\rightarrow\infty$.
If
\[
\nu\left(  x\right)  =\sum_{l\in\mathbb{Z}^{d}}\widehat{\nu}\left(  l\right)
e^{il\cdot x}%
\]
is an accumulation point of $\left(  \left\vert e^{it\Delta}u_{n}\right\vert
^{2}\right)  =\left(  \left\vert u_{n}\right\vert ^{2}\right)  $ then $\nu$ is
more regular than \emph{a priori} expected. It is constant if $d=1$; it is a
trigonometric polynomial if $d=2$, as shown by Jakobson \cite{JakobsonTori97};
and for $d\geq3$ it satisfies (see again \cite{JakobsonTori97}):%
\[
\left\Vert \widehat{\nu}\right\Vert _{l^{d-2}\left(  \mathbb{Z}^{d}\right)
}<\infty,
\]
which is an improvement by two in the summability exponent with respect to the
corresponding estimate (\ref{e:reigen}) for eigenfunctions.

Here we shall deal with sequences of initial data that are more general than
eigenfunctions but still exhibit oscillating behavior around some
characteristic frequencies. More precisely, we shall assume that the sequence
of initial data $\left(  u_{n}\right)  $ satisfies the following condition:

\begin{itemize}
\item[(\textbf{S})] There exists a sequence of positive reals $\left(
h_{n}\right)  $ that tends to zero such that:%
\begin{equation}
\limsup_{n\rightarrow\infty}\sum_{\left\vert l\right\vert <\delta/h_{n}%
}\left\vert \widehat{u_{n}}\left(  l\right)  \right\vert ^{2}\rightarrow
0,\quad\text{as }\delta\rightarrow0^{+}\text{,} \label{e:ho1}%
\end{equation}
and%
\begin{equation}
\limsup_{n\rightarrow\infty}\sum_{\left\vert l\right\vert >R/h_{n}}\left\vert
\widehat{u_{n}}\left(  l\right)  \right\vert ^{2}\rightarrow0,\quad\text{as
}R\rightarrow\infty. \label{e:ho2}%
\end{equation}

\end{itemize}

Above, $\widehat{u_{n}}\left(  l\right)  $, $l\in\mathbb{Z}^{d}$, are the
Fourier coefficients of $u_{n}$. This condition ensures that the $L^{2}$-norm
of $u_{n}$ asymptotically localised on frequencies of the order of $1/h_{n}$.
Condition (\textbf{S}) has been introduced in \cite{GerardSobolev} (in the
context of $\mathbb{R}^{d}$) under the name of $\left(  h_{n}\right)
$-oscillation. We refer the reader to example 2.3 of \cite{GerardSobolev} for
the construction of a sequence weakly converging to zero in $L^{2}\left(
\mathbb{R}^{d}\right)  $ such that the Euclidean version of property
(\textbf{S}) fails.

\begin{theorem}
\label{t:limit}Suppose that $\left(  u_{n}\right)  $ is a bounded sequence in
$L^{2}\left(  \mathbb{T}^{d}\right)  $ that satisfies condition (\textbf{S})
above. Let $\nu\in L^{\infty}\left(  \mathbb{T};L^{1}\left(  \mathbb{T}%
^{d}\right)  \right)  $ be a weak-$\ast$ accumulation point of $\left(
\left\vert e^{it\Delta}u_{n}\right\vert ^{2}\right)  $. If $\widehat{\nu
}\left(  l,s\right)  $, $\left(  l,s\right)  \in\mathbb{Z}^{d}\times
\mathbb{Z}$ are the Fourier coefficients of $\nu\left(  t,x\right)  $ then:%
\[
\left\Vert \widehat{\nu}\right\Vert _{l^{d}\left(  \mathbb{Z}^{d+1}\right)
}<\infty.
\]

\end{theorem}

Therefore, if one imposes condition (\textbf{S}) on the sequence of initial
data, the accumulation points $\nu$ enjoy more regularity than that a priori
expressed by Theorem \ref{t:ld}. The proof of this result is presented in
Section \ref{s:reg}; it is based on two ingredients: estimate (\ref{e:genZyg})
in its matrix-density version; second, the results on the structure of
semiclassical measures for the Schr\"{o}dinger flow in \cite{MaciaTorus,
AnantharamanMacia} that allow us to lower the dimension by one.

If one consider sequences of initial data that satisfy conditions that are
more restrictive than (\textbf{S}) it is possible to show additional
regularity properties on the corresponding weak-$\ast$ accumulation points.
This is the scope of Theorem \ref{t:limitr}, which is stated and proved in the
second part of Section \ref{s:reg}.

\noindent\textbf{Acknowledgment. }The authors would like to thank Patrick
G\'{e}rard for pointing out an inaccuracy in a previous version of this article.


\section{Uniform estimates for the Schr\"{o}dinger flow\label{s:est}}

We first prove Theorem \ref{t:ld}; then we state a straightforward
generalization of this result to solutions to the density-matrix
Schr\"{o}dinger equation that will be used in the proof of Theorem
\ref{t:limit}.

\subsection{Proof of Theorem \ref{t:ld}\label{subsec:proof-ld}}

Let $u\in L^{2}\left(  \mathbb{T}^{d}\right)  $ given by
\[
u\left(  x\right)  =\sum_{k\in\mathbb{Z}^{d}}a_{k}e^{ik\cdot x};
\]
to lighten our writing, we shall use the following normalization of the
$L^{2}$-norm:%
\[
\left\Vert u\right\Vert _{L^{2}\left(  \mathbb{T}^{d}\right)  }^{2}%
:=\sum_{k\in\mathbb{Z}^{d}}\left\vert a_{k}\right\vert ^{2}=\int
_{\mathbb{T}^{d}}\left\vert u\left(  x\right)  \right\vert ^{2}\frac
{dx}{\left(  2\pi\right)  ^{d}}.
\]
Consider the following set of lattice points:%
\[
\mathcal{P}:=\left\{  \left(  k,-\left\vert k\right\vert ^{2}\right)
\;:\;k\in\mathbb{Z}^{d}\right\}  \subset\mathbb{Z}^{d+1};
\]
this is precisely the set of lattice points contained in the paraboloid of
$\mathbb{R}^{d+1}$ obtained as the graph of $-\left\vert x\right\vert ^{2}$
for $x\in\mathbb{R}^{d}$. We shall denote a generic point in $\mathbb{Z}%
^{d+1}$ as $\left(  k,n\right)  $ with $k\in\mathbb{Z}^{d}$ and $n\in
\mathbb{Z}$.

The orbit of the Schr\"{o}dinger flow corresponding to $u$ is:%
\begin{equation}
e^{it\Delta}u\left(  x\right)  =\sum_{k\in\mathbb{Z}^{d}}a_{k}e^{i\left(
k\cdot x-\left\vert k\right\vert ^{2}t\right)  }=\sum_{\left(  k,n\right)
\in\mathcal{P}}a_{k}e^{i\left(  k,n\right)  \cdot\left(  x,t\right)  },
\label{e:schrod}%
\end{equation}
and,%
\[
\left\vert e^{it\Delta}u\left(  x\right)  \right\vert ^{2}=\sum_{\left(
l,s\right)  \in\mathbb{Z}^{d+1}}b_{u}\left(  l,s\right)  e^{i\left(
l,s\right)  \cdot\left(  x,t\right)  },
\]
where, for $\left(  l,s\right)  \in\mathbb{Z}^{d+1}$ we have set:%
\begin{equation}
b_{u}\left(  l,s\right)  :=\sum_{\substack{k-j=l,\\\left\vert j\right\vert
^{2}-\left\vert k\right\vert ^{2}=s}}a_{k}\overline{a_{j}}=\sum
_{\substack{\left(  k,n\right)  -\left(  j,m\right)  =\left(  l,s\right)
,\\\left(  k,n\right)  ,\left(  j,m\right)  \in\mathcal{P}}}a_{k}%
\overline{a_{j}}; \label{e:msquare}%
\end{equation}
note that the range of the sum above might be empty, in which case we set
$b_{u}\left(  l,s\right)  :=0$. Clearly
\begin{equation}
b_{u}\left(  0,0\right)  =\left\Vert u\right\Vert _{L^{2}\left(
\mathbb{T}^{d}\right)  }^{2},\quad b_{u}\left(  0,s\right)  =0\text{ for }%
s\in\mathbb{Z}\setminus\left\{  0\right\}  , \label{e:trivcoefs}%
\end{equation}

We start by making some elementary geometric remarks. Let $j\in\mathbb{Z}^{d}$
and suppose that there exist $\left(  l,s\right)  \in\mathbb{Z}^{d+1}%
\setminus\left\{  \left(  0,0\right)  \right\}  $ and $k\in\mathbb{Z}^{d}$
such that $k-j=l$ and $\left\vert j\right\vert ^{2}-\left\vert k\right\vert
^{2}=s$. This is the same as saying that $\left(  l,s\right)  $ is a chord of
the discrete paraboloid $\mathcal{P}$ with origin at $(j,-\left\vert
j\right\vert ^{2})$. We can rewrite the condition on the squares of the
lengths of $k$ and $j$ as:%
\begin{equation}
s=-l\cdot\left(  k+j\right)  =-l\cdot\left(  l+2j\right)  =-\left\vert
l\right\vert ^{2}-2l\cdot j. \label{e:s}%
\end{equation}
Therefore, the set of all $j\in\mathbb{Z}^{d}$ such that $(j,-\left\vert
j\right\vert ^{2})$ is the origin of a chord $\left(  l,s\right)  $ on the
paraboloid $\mathcal{P}$ is equal to
\[
H_{\left(  l,s\right)  }:=\left\{  j\in\mathbb{Z}^{d}\;:\;j\cdot\frac
{l}{\left\vert l\right\vert }=-\frac{1}{2}\left(  \frac{s}{\left\vert
l\right\vert }+\left\vert l\right\vert \right)  \right\}  ,
\]
which is the set of lattice points that lie on a certain hyperplane of
$\mathbb{R}^{d}$ which is orthogonal to $l$ and at whose distance from the
origin is determined by $\left\vert l\right\vert $ and $s$. Note in particular
that
\[
H_{\left(  l,s\right)  }\cap H_{\left(  l,s^{\prime}\right)  }\neq
\emptyset\;\Leftrightarrow\;s=s^{\prime}.
\]
We summarize the preceding geometric discussion as:

\begin{center}
\emph{Let }$\left(  k,n\right)  ,\left(  j,m\right)  \in\mathcal{P}$
\emph{and} $\left(  l,s\right)  \in\mathbb{Z}^{d+1}$ \emph{with }$l\neq0$.
\emph{That }$\left(  k,n\right)  -\left(  j,m\right)  =\left(  l,s\right)  $
\emph{is equivalent to }$j\in H_{\left(  l,s\right)  }$ \emph{where} $s$
\emph{is given by (\ref{e:s}) and }$k=j+l.$
\end{center}

Now, by (\ref{e:trivcoefs}),
\begin{equation}
\sum_{\left(  l,s\right)  \in\mathbb{Z}^{d+1}}\left\vert b_{u}\left(
l,s\right)  \right\vert ^{d+1}=\left\Vert u\right\Vert _{L^{2}\left(
\mathbb{T}^{d}\right)  }^{2\left(  d+1\right)  }+\sum_{\left(  l,s\right)
\in\mathbb{Z}^{d+1},l\neq0}\left\vert b_{u}\left(  l,s\right)  \right\vert
^{d+1}. \label{e:bd}%
\end{equation}
Our goal will be to estimate the second term in the above sum by $\left\Vert
u\right\Vert _{L^{2}\left(  \mathbb{T}^{d}\right)  }^{2\left(  d+1\right)  }$.
We have proved above that, for $l\not =0$, we can rewrite (\ref{e:msquare})
as:%
\[
b_{u}\left(  l,s\right)  =\sum_{j\in H_{\left(  l,s\right)  }}a_{j+l}%
\overline{a_{j}},
\]
and therefore,
\begin{equation}
\left\vert b_{u}\left(  l,s\right)  \right\vert ^{d+1}\leq\sum_{j_{1}%
,...,j_{d+1}\in H_{\left(  l,s\right)  }}%
{\displaystyle\prod_{\sigma=1}^{d+1}}
\left\vert a_{j_{\sigma}+l}a_{j_{\sigma}}\right\vert . \label{e:sumb}%
\end{equation}

Let $\delta\geq0$ be an integer. For any $r\in\left\{  0,1,..,\delta
-1\right\}  $ denote by $\mathcal{V}_{r}^{\delta}$ the set of the $\left(
j_{1},...,j_{\delta+1}\right)  \in\mathbb{Z}^{d\left(  \delta+1\right)  }$
such that $j_{1},...,j_{\delta+1}$ span an affine variety in $\mathbb{R}^{d}$
of dimension $r$. We can rewrite (\ref{e:sumb}) as:%
\[
\left\vert b_{u}\left(  l,s\right)  \right\vert ^{d+1}\leq\sum_{r=0}^{d-1}%
\sum_{\left(  j_{1},...,j_{d+1}\right)  \in\left(  H_{\left(  l,s\right)
}\right)  ^{d+1}\cap\mathcal{V}_{r}^{d}}%
{\displaystyle\prod_{\sigma=1}^{d+1}}
\left\vert a_{j_{\sigma}+l}a_{j_{\sigma}}\right\vert \text{.}%
\]
Therefore, in order to have a bound for $\left\Vert b_{u}\right\Vert
_{l^{d+1}}^{d+1}$ it suffices to estimate, for each $r=0,1,...,d-1$, the term:%
\begin{equation}
T_{r,d}\left(  u\right)  :=\sum_{\substack{\left(  l,s\right)  \in
\mathbb{Z}^{d+1},\\l\neq0}}\sum_{\left(  j_{1},...,j_{d+1}\right)  \in\left(
H_{\left(  l,s\right)  }\right)  ^{d+1}\cap\mathcal{V}_{r}^{d}}%
{\displaystyle\prod_{\sigma=1}^{d+1}}
\left\vert a_{j_{\sigma}+l}a_{j_{\sigma}}\right\vert , \label{e:rfixed}%
\end{equation}
which involves only indices $j_{1},...,j_{d+1}$ that span an affine variety of
dimension $r$.

This is a consequence of the next lemma, which is slightly more general than
what is needed at this point, but that will also be needed in the proof of
Theorem \ref{t:lddm} in the next paragraph.

\begin{lemma}
\label{l:estT}Let $\delta$ and $r$ be integers with $0\leq\delta\leq d$ and
$0\leq r\leq\delta-1$. For every $u\in L^{2}\left(  \mathbb{T}^{d}\right)  $
given by $u\left(  x\right)  =\sum_{k\in\mathbb{Z}^{d}}a_{k}e^{ik\cdot x}$
define:
\begin{equation}
T_{r,\delta}\left(  u\right)  :=\sum_{\substack{\left(  l,s\right)
\in\mathbb{Z}^{d+1},\\l\neq0}}\sum_{\left(  j_{1},...,j_{\delta+1}\right)
\in\left(  H_{\left(  l,s\right)  }\right)  ^{\delta+1}\cap\mathcal{V}%
_{r}^{\delta}}%
{\displaystyle\prod_{\sigma=1}^{\delta+1}}
\left\vert a_{j_{\sigma}+l}a_{j_{\sigma}}\right\vert . \label{e:fdeltafixed}%
\end{equation}
Then%
\[
T_{r,\delta}\left(  u\right)  \leq\dbinom{\delta}{r}\left\Vert u\right\Vert
_{L^{2}\left(  \mathbb{T}^{d}\right)  }^{2\left(  \delta+1\right)  }.
\]

\end{lemma}

\begin{proof}
The case $r=0$ is simple to estimate. Notice that given $\left(  k,j\right)
\in\mathbb{Z}^{d}\times\mathbb{Z}^{d}$ with $k\neq j$ there exist a unique
$\left(  l,s\right)  \in\mathbb{Z}^{d+1}$ with $l\neq0$ such that $j\in
H_{\left(  l,s\right)  }$ and $k=j+l$. Therefore:%
\begin{align*}
T_{0,\delta}\left(  u\right)   &  =\sum_{\substack{\left(  l,s\right)
\in\mathbb{Z}^{d+1}\\l\neq0}}\sum_{j\in H_{\left(  l,s\right)  }}\left\vert
a_{j+l}a_{j}\right\vert ^{\delta+1}\\
&  =\sum_{\substack{\left(  k,j\right)  \in\mathbb{Z}^{d}\times\mathbb{Z}%
^{d}\\k\neq j}}\left\vert a_{k}a_{j}\right\vert ^{\delta+1}\leq\left(
\sum_{j\in\mathbb{Z}^{d}}\left\vert a_{j}\right\vert ^{\delta+1}\right)
^{2}\leq\left\Vert u\right\Vert _{L^{2}\left(  \mathbb{T}^{d}\right)
}^{2\left(  \delta+1\right)  }.
\end{align*}
Now suppose $1\leq r\leq\delta-1$ and consider a fixed summand in
(\ref{e:fdeltafixed}):%
\begin{equation}%
{\displaystyle\prod_{\sigma=1}^{\delta+1}}
\left\vert a_{j_{\sigma}+l}a_{j_{\sigma}}\right\vert ,\label{e:prod}%
\end{equation}
corresponding to some $\left(  j_{1},...,j_{\delta+1}\right)  \in
\mathcal{V}_{r}^{\delta}$; denote by $\mathcal{L}$ the $r$-dimensional affine
variety spanned by the $j_{\sigma}$'s. Choose the least indices $1=\alpha
_{1}<...<\alpha_{r+1}\leq\delta+1$ such that $j_{\alpha_{1}},...,j_{\alpha
_{r+1}}$ span $\mathcal{L}$, and denote by $\beta_{1}<...<\beta_{\delta-r}$
the remaining indices. We now estimate (\ref{e:prod}) by:%
\begin{equation}
\frac{1}{2}\left(
{\displaystyle\prod_{i=1}^{r+1}}
\left\vert a_{j_{\alpha_{i}}}\right\vert ^{2}%
{\displaystyle\prod_{i=1}^{\delta-r}}
\left\vert a_{j_{\beta_{i}}+l}\right\vert ^{2}+%
{\displaystyle\prod_{i=1}^{r+1}}
\left\vert a_{j_{\alpha_{i}}+l}\right\vert ^{2}%
{\displaystyle\prod_{i=1}^{\delta-r}}
\left\vert a_{j_{\beta_{i}}}\right\vert ^{2}\right)  .\label{e:ts}%
\end{equation}
We shall perform our analysis focusing on the first summand. It will be clear
that the second summand can be dealt with in a completely analogous manner.

Given $\left(  j_{1},...,j_{\delta+1}\right)  \in\mathcal{V}_{r}^{\delta}$ we
define $\left(  \tilde{j}_{1},...,\tilde{j}_{\delta+1}\right)  \in
\mathcal{V}_{r+1}^{\delta}$ as follows: set $\tilde{j}_{i}=j_{\alpha_{i}}$ for
$i=1,...,r+1$ and $\tilde{j}_{i}=j_{\beta_{i}}+l$ for $i=r+2,...,\delta+1$.
Thus, we can put the first summand in (\ref{e:ts}) as:%
\[%
{\displaystyle\prod_{i=1}^{\delta+1}}
\left\vert a_{\tilde{j}_{i}}\right\vert ^{2},\quad\text{for some }\left(
\tilde{j}_{1},...,\tilde{j}_{\delta+1}\right)  \in\mathcal{V}_{r+1}^{\delta}.
\]
In order to estimate $T_{r,\delta}\left(  u\right)  $ by a sum of terms of
this form over indices in $\mathcal{V}_{r+1}^{\delta}$ we should take into
account the following:

\noindent1. A $\left(  \delta+1\right)  $-tuple $\left(  \tilde{j}%
_{1},...,\tilde{j}_{\delta+1}\right)  \in\mathcal{V}_{r+1}^{\delta}$ is
obtained from $\left(  j_{1},...,j_{\delta+1}\right)  \in\mathcal{V}%
_{r}^{\delta}$ by applying a permutation that maps monotonically the indices
$1=\alpha_{1}<...<\alpha_{r+1}$ into $1,..,r+1$ and the indices $\beta
_{1},...,\beta_{\delta-r}$ into $r+2,...,\delta+1$. Note that there are
$\dbinom{\delta}{r}$ such permutations.

\noindent2. Suppose $\left(  \tilde{j}_{1},...,\tilde{j}_{\delta+1}\right)  $
is obtained from $\left(  j_{1},...,j_{\delta+1}\right)  \in\mathcal{V}%
_{r}^{\delta}$. Then there exists a unique $\left(  l,s\right)  \in
\mathbb{Z}^{d+1}$, $l\neq0$, such that $\left(  j_{1},...,j_{\delta+1}\right)
\in\left(  H_{\left(  l,s\right)  }\right)  ^{\delta+1}$. This is due to the
fact that that $l$ is a direction contained in the $\left(  r+1\right)
$-dimensional affine variety spanned by $\tilde{j}_{1},...,\tilde{j}_{r+2}$
and that $l$ must be orthogonal to
\[
\mathcal{L}:=\operatorname{span}\left\{  j_{1},...,j_{\delta+1}\right\}
=\operatorname{span}\left\{  \tilde{j}_{1},...,\tilde{j}_{r+1}\right\}  .
\]
Therefore, it must coincide with the orthogonal projection of, say, $\tilde
{j}_{r+2}-\tilde{j}_{1}$ onto the variety that is orthogonal to $\mathcal{L}$
and contains $j_{1}=\tilde{j}_{1}$. Once $l$ is determined uniquely so is $s$,
by (\ref{e:s}).

As a consequence of this, we deduce the bound:%
\begin{align*}
&  \sum_{\substack{\left(  l,s\right)  \in\mathbb{Z}^{d+1},\\l\neq0}%
}\sum_{\left(  j_{1},...,j_{\delta+1}\right)  \in\left(  H_{\left(
l,s\right)  }\right)  ^{\delta+1}\cap\mathcal{V}_{r}^{\delta}}%
{\displaystyle\prod_{i=1}^{r+1}}
\left\vert a_{j_{\alpha_{i}}}\right\vert ^{2}%
{\displaystyle\prod_{i=1}^{\delta-r}}
\left\vert a_{j_{\beta_{i}}+l}\right\vert ^{2}\\
&  \leq\dbinom{\delta}{r}\sum_{\left(  \tilde{j}_{1},...,\tilde{j}_{\delta
+1}\right)  \in\mathcal{V}_{r+1}^{\delta}}%
{\displaystyle\prod_{\sigma=1}^{\delta+1}}
\left\vert a_{\tilde{j}_{\sigma}}\right\vert ^{2}\leq\dbinom{\delta}%
{r}\left\Vert u\right\Vert _{L^{2}\left(  \mathbb{T}^{d}\right)  }^{2\left(
\delta+1\right)  }.
\end{align*}
The same estimate can be established for the sums corresponding to the second
summands in (\ref{e:ts}); therefore:
\[
T_{r,\delta}\left(  u\right)  \leq\dbinom{\delta}{r}\left\Vert u\right\Vert
_{L^{2}\left(  \mathbb{T}^{d}\right)  }^{2\left(  \delta+1\right)  }.
\]

\end{proof}

Now, applying Lemma \ref{l:estT} with $\delta=d$ we obtain:%
\[
\sum_{\left(  l,s\right)  \in\mathbb{Z}^{d+1}}\left\vert b_{u}\left(
l,s\right)  \right\vert ^{d+1}=\left\Vert u\right\Vert _{L^{2}\left(
\mathbb{T}^{d}\right)  }^{2\left(  d+1\right)  }+\sum_{r=0}^{d-1}%
T_{r,d}\left(  u\right)  \leq2^{d}\left\Vert u\right\Vert _{L^{2}\left(
\mathbb{T}^{d}\right)  }^{2\left(  d+1\right)  },
\]
which concludes the proof of Theorem \ref{t:ld}.

\subsection{Estimates for the density-matrix Schr\"{o}dinger equation}

Denote by $\mathcal{L}_{+}^{1}$ the set of symmetric, non-negative,
trace-class operators on $L^{2}\left(  \mathbb{T}^{d}\right)  $. Every
operator $A\in\mathcal{L}_{+}^{1}$ is defined by an integral kernel $\rho\in
L^{2}\left(  \mathbb{T}^{d}\times\mathbb{T}^{d}\right)  $,%
\[
\rho\left(  x,y\right)  =\sum_{k,j}\widehat{\rho}_{k,j}e^{ik\cdot
x}e^{-ij\cdot y}.
\]
The self-adjointness of $A$ implies that $\rho$ is symmetric:%
\[
\rho\left(  x,y\right)  =\overline{\rho\left(  y,x\right)  },
\]
and from the fact that $A$ is non-negative and trace-class we deduce that%
\[
t_{\rho}\left(  x\right)  :=\sum_{k,j}\widehat{\rho}_{k,j}e^{i\left(
k-j\right)  \cdot x}%
\]
is a non-negative function in $L^{1}\left(  \mathbb{T}^{d}\right)  $.

When no confusion arises, we shall use $\rho$ to refer both to a generic
operator in $\mathcal{L}_{+}^{1}$ and to its integral kernel. Note that, since
operators $\rho\in\mathcal{L}_{+}^{1}$ are non-negative and symmetric, their
trace norm is given by:%
\[
\left\Vert \rho\right\Vert _{\mathcal{L}^{1}}:=\operatorname*{Tr}\left(
\rho\right)  =\left\Vert t_{\rho}\right\Vert _{L^{1}\left(  \mathbb{T}%
^{d}\right)  }.
\]
Moreover, if $\rho\left(  x,y\right)  =u\left(  x\right)  \overline{u\left(
y\right)  }$ is the orthogonal projection on a function $u\in L^{2}\left(
\mathbb{T}^{d}\right)  $ then $\rho\in\mathcal{L}_{+}^{1}$ and $t_{\rho
}=\left\vert u\right\vert ^{2}$. Usually, one refers to the elements of
$\mathcal{L}_{+}^{1}$ as \emph{density matrices}.

The Schr\"{o}dinger equation for density matrices (also known as the
Heisenberg-Von Neumann equation) is:%
\begin{equation}
i\partial_{t}A\left(  t\right)  +\left[  \Delta,A\left(  t\right)  \right]
=0,\quad A\left(  0\right)  =\rho\in\mathcal{L}_{+}^{1}, \label{e:H-N}%
\end{equation}
where $\left[  \cdot,\cdot\right]  $ denotes the commutator bracket. Clearly,%
\[
A\left(  t\right)  =e^{it\Delta}\rho e^{-it\Delta},
\]
so, as $e^{it\Delta}$ is unitary, $A\left(  t\right)  \in\mathcal{L}_{+}^{1}$
and $\left\Vert A\left(  t\right)  \right\Vert _{\mathcal{L}^{1}}=\left\Vert
\rho\right\Vert _{\mathcal{L}^{1}}$ for every $t\in\mathbb{R}$ . Note that
$A\left(  t\right)  $ is $2\pi\mathbb{Z}$-periodic in $t$.

When $\rho\left(  x,y\right)  =u\left(  x\right)  \overline{u\left(  y\right)
}$, the integral kernel of $e^{it\Delta}\rho e^{-it\Delta}$ is $\psi\left(
x,t\right)  \overline{\psi\left(  y,t\right)  }$ where $\psi\left(
\cdot,t\right)  :=e^{it\Delta}u$, and $t_{e^{it\Delta}\rho e^{-it\Delta}%
}=\left\vert e^{it\Delta}u\right\vert ^{2}$. Therefore, the dynamics of
equation (\ref{e:H-N}) reduce to those of (\ref{def:psi}) in that case.

For a given $\rho\in\mathcal{L}_{+}^{1}$, we shall denote by $b_{\rho}$ the
Fourier coefficients of $t_{e^{it\Delta}\rho e^{-it\Delta}}$:%
\[
t_{e^{it\Delta}\rho e^{-it\Delta}}\left(  x\right)  =\sum_{\left(  l,s\right)
\in\mathbb{Z}^{d}\times\mathbb{Z}}b_{\rho}\left(  l,s\right)  e^{i\left(
l,s\right)  \cdot\left(  x,t\right)  }.
\]

We now introduce some special classes of initial data in $\mathcal{L}_{+}^{1}%
$. Let $\Lambda$ be a submodule of $\mathbb{Z}^{d}$, let $\operatorname*{rk}%
\Lambda$ denote the rank of $\Lambda$. We define $\mathcal{L}_{+}^{1}\left(
\Lambda\right)  $ as the set consisting of the $\rho\in\mathcal{L}_{+}^{1}$
such that the integral kernel of $\rho$ is invariant by translations in
directions in $\Lambda^{\perp}$ (this is the linear subspace orthogonal to
$\Lambda$); in other words:%
\begin{equation}
\rho\left(  x+v,y+w\right)  =\rho\left(  x,y\right)  ,\qquad\forall
v,w\in\Lambda^{\perp}. \label{e:sym}%
\end{equation}
Clearly, the classes $\mathcal{L}_{+}^{1}\left(  \Lambda\right)  $ are
invariant by the dynamics of (\ref{e:H-N}). The following result holds.

\begin{theorem}
\label{t:lddm}Let $\Lambda\subset\mathbb{Z}^{d}$ be a submodule with
$\operatorname*{rk}\Lambda>0$. Then there exists a constant
$C_{\operatorname*{rk}\Lambda}>0$ such that:
\[
\left\Vert b_{\rho}\right\Vert _{l^{\operatorname*{rk}\Lambda+1}\left(
\mathbb{Z}^{d+1}\right)  }\leq C_{\operatorname*{rk}\Lambda}\left\Vert
t_{\rho}\right\Vert _{L^{1}\left(  \mathbb{T}^{d}\right)  },
\]
for every $\rho\in\mathcal{L}_{+}^{1}\left(  \Lambda\right)  $.
\end{theorem}

\begin{proof}
Let $\left(  u_{n}\right)  $ be an orthonormal basis of $L^{2}\left(
\mathbb{T}^{d}\right)  $ consisting of eigenfunctions of $\rho$. One has%
\[
\rho\left(  x,y\right)  =\sum_{i=1}^{\infty}\lambda_{n}u_{n}\left(  x\right)
\overline{u_{n}\left(  y\right)  },
\]
with $\lambda_{n}\geq0$, $n\in\mathbb{N}$, and $\sum_{n=1}^{\infty}\lambda
_{n}=\operatorname{Tr}\left(  \rho\right)  =\left\Vert t_{\rho}\right\Vert
_{L^{1}\left(  \mathbb{T}^{d}\right)  }$. Clearly,%
\begin{equation}
b_{\rho}=\sum_{n=1}^{\infty}\lambda_{n}b_{u_{n}}; \label{e:buj}%
\end{equation}
moreover, since $\rho$ satisfies (\ref{e:sym}) and $\Lambda$ is a submodule of
$\mathbb{Z}^{d}$ necessarily each $u_{n}$ is of the form:%
\[
u_{n}\left(  x\right)  =\sum_{k\in\Lambda}a_{k}^{n}e^{ik\cdot x},
\]
and $b_{u_{n}}\left(  l,s\right)  =0$ whenever $l\notin\Lambda$. Note moreover
that for $l\in\Lambda\setminus\left\{  0\right\}  $ one has that $\Lambda\cap
H_{\left(  l,s\right)  }$ spans an affine variety of dimension strictly less
than $\operatorname*{rk}\Lambda$ (since $\left\{  nl\right\}  _{n\in
\mathbb{Z}}$ is not contained in $H_{\left(  l,s\right)  }$). This implies
that $\left(  \Lambda\cap H_{\left(  l,s\right)  }\right)
^{\operatorname*{rk}\Lambda+1}\cap\mathcal{V}_{\operatorname*{rk}\Lambda
}^{\operatorname*{rk}\Lambda}=\emptyset$ and%
\[
\left\vert b_{u_{n}}\left(  l,s\right)  \right\vert ^{\operatorname*{rk}%
\Lambda+1}\leq\sum_{r=0}^{\operatorname*{rk}\Lambda-1}\sum_{\left(
j_{1},...,j_{\operatorname*{rk}\Lambda+1}\right)  \in\left(  \Lambda\cap
H_{\left(  l,s\right)  }\right)  ^{\operatorname*{rk}\Lambda+1}\cap
\mathcal{V}_{r}^{\operatorname*{rk}\Lambda}}%
{\displaystyle\prod_{\sigma=1}^{\operatorname*{rk}\Lambda+1}}
\left\vert a_{j_{\sigma}+l}^{n}a_{j_{\sigma}}^{n}\right\vert .
\]
With this in mind, one can apply Lemma \ref{l:estT} with $\delta
=\operatorname*{rk}\Lambda$ to conclude, as in Theorem \ref{t:ld}:%
\[
\sum_{\substack{\left(  l,s\right)  \in\mathbb{Z}^{d+1}\\l\in\Lambda
}}\left\vert b_{u_{n}}\left(  l,s\right)  \right\vert ^{\operatorname*{rk}%
\Lambda+1}=1+\sum_{r=0}^{\operatorname*{rk}\Lambda-1}T_{r,\operatorname*{rk}%
\Lambda}\left(  u_{n}\right)  \leq2^{\operatorname*{rk}\Lambda}.
\]
Therefore, using (\ref{e:buj}) we obtain the estimate:%
\[
\left\Vert b_{\rho}\right\Vert _{l^{\operatorname*{rk}\Lambda+1}\left(
\mathbb{Z}^{d+1}\right)  }\leq\sum_{n=0}^{\infty}\lambda_{n}\left\Vert
b_{u_{n}}\right\Vert _{l^{\operatorname*{rk}\Lambda+1}\left(  \mathbb{Z}%
^{d+1}\right)  }\leq C_{\operatorname*{rk}\Lambda}\sum_{n=0}^{\infty}%
\lambda_{n}=C_{\operatorname*{rk}\Lambda}\left\Vert t_{\rho}\right\Vert
_{L^{1}\left(  \mathbb{T}^{d}\right)  }.
\]

\end{proof}

\section{Regularity of the limits\label{s:reg}}

We start by giving the proof of Theorem \ref{t:limit}, which relies on Theorem
\ref{t:lddm} and the results from \cite{AnantharamanMacia}.

\subsection{Proof of Theorem \ref{t:limit}}

Without loss of generality, we can assume that (\ref{e:wc}) holds for the
whole sequence $\left(  u_{n}\right)  $. Let $\left(  h_{n}\right)  $ the
sequence appearing in property (\textbf{S}). Denote by $\nu$ the weak-$\ast$
limit of $\left(  \left\vert e^{it\Delta}u_{n}\right\vert ^{2}\right)  $.
Theorem 3 in \cite{AnantharamanMacia} ensures that,
\begin{equation}
\nu\left(  x,t\right)  =\sum_{\Lambda}t_{e^{it\Delta}\rho_{\Lambda
}e^{-it\Delta}}\left(  x\right)  , \label{e:snu}%
\end{equation}
where the sum ranges over all (primitive) submodules of $\mathbb{Z}^{d}$, and
the operators $\rho_{\Lambda}\in\mathcal{L}_{+}^{1}\left(  \Lambda\right)  $
only depend on the sequence of initial data $\left(  u_{n}\right)  $. Because
the functions $t_{\rho_{\Lambda}}$ are non-negative, we have:%
\[
\sum_{\Lambda}\left\Vert t_{\rho_{\Lambda}}\right\Vert _{L^{1}\left(
\mathbb{T}^{d}\right)  }=\sum_{\Lambda}\left\Vert t_{e^{it\Delta}\rho
_{\Lambda}e^{-it\Delta}}\right\Vert _{L^{1}\left(  \mathbb{T}^{d}\right)
}=\nu\left(  \mathbb{T}^{d},t\right)  \leq\limsup_{n\rightarrow\infty
}\left\Vert u_{n}\right\Vert _{L^{2}\left(  \mathbb{T}^{d}\right)  }^{2}.
\]
We claim that if hypothesis (\textbf{S}) holds, then $\rho_{\mathbb{Z}^{d}}%
=0$. In that case, only submodules of rank strictly less than $d$ appear in
the sum (\ref{e:snu}). Therefore, we can apply Theorem \ref{t:lddm} and find:%
\[
\left\Vert \widehat{\nu}\right\Vert _{l^{d}\left(  \mathbb{Z}^{d+1}\right)
}\leq\sum_{\Lambda}\left\Vert b_{\rho_{\Lambda}}\right\Vert _{l^{d}\left(
\mathbb{Z}^{d+1}\right)  }\leq C\sum_{\Lambda}\left\Vert t_{\rho_{\Lambda}%
}\right\Vert _{L^{1}\left(  \mathbb{T}^{d}\right)  }<\infty.
\]
We finally show that $\rho_{\mathbb{Z}^{d}}=0$. Hypothesis (\textbf{S})
implies that any limit $\mu_{0}\in\mathcal{M}_{+}\left(  \mathbb{T}^{d}%
\times\mathbb{R}^{d}\right)  $ of the sequence of Wigner distributions
$w_{u_{n}}^{h_{n}}$ (in the notation of \cite{AnantharamanMacia}) satisfies:
\[
\mu_{0}\left(  \mathbb{T}^{d}\times\left\{  0\right\}  \right)  =0.
\]
The reader may find the standard argument in, for instance,
\cite{GerardSobolev}. One can then apply Corollary 30 in
\cite{AnantharamanMacia} and conclude the proof.

\subsection{Additional regularity results}

It is possible to replace (\textbf{S}) by a family of stronger conditions that
ensure additional regularity on the weak-$\ast$ accumulation points of orbits
of the Schr\"{o}dinger flow. In order to introduce them we must recall some
notations from \cite{AnantharamanMacia}.

Denote by $\Omega_{j}\subset\mathbb{R}^{d}$, for $j=0,...,d$, the set of
resonant vectors of order exactly $j$, that is:%
\[
\Omega_{j}:=\left\{  \xi\in\mathbb{R}^{d}:\operatorname*{rk}\Lambda_{\xi
}=d-j\right\}  ,
\]
where $\Lambda_{\xi}:=\left\{  k\in\mathbb{Z}^{d}:k\cdot\xi=0\right\}  $. A
classical result shows that $\xi\in\Omega_{j}$ if and only if for any
$x\in\mathbb{T}^{d}$ the geodesic $\tau\mapsto x+\tau\xi$ is dense in a
subtorus of $\mathbb{T}^{d}$ of dimension $j$.

For each $r=1,...,d+1$ we introduce the following condition (\textbf{S}$_{r}$)
on the sequence of initial data $\left(  u_{n}\right)  $:

\begin{itemize}
\item[(\textbf{S}$_{r}$)] There exists a sequence of positive reals $\left(
h_{n}\right)  $ tending to zero such that (\ref{e:ho2}) holds and every
accumulation point $\mu_{0}\in\mathcal{M}_{+}\left(  \mathbb{T}^{d}%
\times\mathbb{R}^{d}\right)  $ of the sequence of Wigner distributions
$\left(  w_{u_{n}}^{h_{n}}\right)  $ (as defined in \cite{AnantharamanMacia})
satisfies:%
\[
\mu_{0}\left(  \mathbb{T}^{d}\times%
{\displaystyle\bigcup_{j<r}}
\Omega_{j}\right)  =0.
\]

\end{itemize}

Note that the condition (\textbf{S}) is equivalent to (\textbf{S}$_{1}$);
moreover, (\textbf{S}$_{d+1}$) is equivalent to the fact that $\left(
u_{n}\right)  $ converges strongly to $0$ in $L^{2}\left(  \mathbb{T}%
^{d}\right)  $. Condition (\textbf{S}$_{r}$) roughly states that no fraction
of the $L^{2}$-norm of the $u_{n}$ concentrates on resonances of order stricly
less than $r$ as $n\rightarrow\infty$.

In \cite{MaciaAv} it is shown that whenever (\textbf{S}$_{d}$) holds, we have
that every weak-$\ast$ accumulation point is constant.

The following generalization of Theorem \ref{t:limit} is true.

\begin{theorem}
\label{t:limitr}Suppose that $\left(  u_{n}\right)  $ is a bounded sequence in
$L^{2}\left(  \mathbb{T}^{d}\right)  $ that satisfies condition (\textbf{S}%
$_{r}$) for some $r=1,...,d-1$. Let $\nu\in L^{\infty}\left(  \mathbb{T}%
;L^{1}\left(  \mathbb{T}^{d}\right)  \right)  $ be a weak-$\ast$ accumulation
point of $\left(  \left\vert e^{it\Delta}u_{n}\right\vert ^{2}\right)  $.
Then:%
\[
\left\Vert \widehat{\nu}\right\Vert _{l^{d+1-r}\left(  \mathbb{Z}%
^{d+1}\right)  }<\infty.
\]

\end{theorem}

The proof of this result follows exactly the lines of that of Theorem
\ref{t:limit}; except that hypothesis (\textbf{S}$_{r}$) is now used, via
Corollary 30 in \cite{AnantharamanMacia}, to ensure that $\rho_{\Lambda}=0$
for every primitive submodule $\Lambda\subset\mathbb{Z}^{d}$ of rank strictly
greater than $d-r$.


\end{document}